\documentclass[a4paper,11pt,reqno]{amsart}
\usepackage[utf8]{inputenc}
\usepackage{amsfonts}
\usepackage{amsmath}
\usepackage{amssymb}
\usepackage{amsthm}
\usepackage{verbatim}
\usepackage{color}
\usepackage{hyperref}
\usepackage{epstopdf}

\usepackage{thmtools}
\usepackage{theoremref}
\theoremstyle{plain}
\newtheorem{theorem}{Theorem}[section]
\newtheorem{lemma}[theorem]{Lemma}
\newtheorem{corollary}[theorem]{Corollary}

\theoremstyle{definition}
\newtheorem{definition}[theorem]{Definition}

\numberwithin{equation}{section}

\title[Twisted Moments of Dirichlet $L$-functions]{A Note On Twisted Moments of Dirichlet $L$-functions}

\author{J. C. Andrade}
\address{Department of Mathematics, University of Exeter, Exeter, EX4 4QF, United Kingdom}
\email{j.c.andrade@exeter.ac.uk}

\author{J. MacMillan}
\address{Department of Mathematics, University of Exeter, Exeter, EX4 4QF, United Kingdom}
\email{jm1015@exeter.ac.uk}

\date{\today}

\subjclass[2010]{Primary 11M38; Secondary 11M06, 11T06, 11T55}
\keywords{twisted moments, Dirichlet $L$-functions, function fields}

\begin{document}

\begin{abstract}
   In this paper, we establish an asymptotic formula for the twisted second moments of Dirichlet $L$-functions with one and two twists when averaged over all primitive Dirichlet characters of modulus $R$, where $R$ is a monic polynomial in $\mathbb{F}_q[T]$. The main result in this paper generalizes the work of Djankovi\'c [`The reciprocity law for the twisted second moment of Dirichlet $L$-functions over rational function fields', Bull. Aust. Math. Soc. \textbf{98} (2018), no. 3, 382--388]. 
\end{abstract}

\maketitle

\section{Introduction}

It is well-known that the study of moments of the Riemann zeta-function and $L$-functions is an important 
topic in analytic number theory. It can be even argued that great part of 
research in analytic number theory in the last century have been guided and motivated 
by this topic.

Applications of moments of $L$-functions appears more notably in 
the Lindel\"{o}f hypothesis, but also when studying proportions of zeros satisfying the
Riemann hypothesis and nonvanishing at the central point of families of $L$-functions. For some 
these applications is important to understand not only the moments of $L$-functions but also
what is known as \textit{twisted moments}. 

Let $\chi$ be a Dirichlet character modulo $p$, where $p$ is a prime number. The problem 
is then to obtain a formula for 
\begin{equation}\label{eq:S}
\mathcal{S}(p,h):=\sideset{}{^*}\sum_{\chi(\bmod p)}\left|L\left(\frac{1}{2},\chi\right)\right|^{2}\chi(h),
\end{equation}
where $h$ is a fixed prime number and the $\ast$ indicates a summation over all primitive Dirichlet characters modulo $p$. 
With this notation, Conrey \cite[Theorem 10]{Con} proved the following.

\begin{theorem}[Conrey \cite{Con}]
For primes $p,h$ with $2\leq h<p$ we have that
\begin{align*}
\mathcal{S}(p,h)&=\frac{p^{1/2}}{h^{1/2}}\mathcal{S}(h,-p)+\frac{p}{h^{1/2}}\left(\log\frac{p}{h}+\gamma-\log(8\pi)\right)+\zeta\left(\frac{1}{2}\right)^{2}p^{1/2}\\
&+O\left(h+\log p+\frac{p^{1/2}}{h^{1/2}}\log p\right),
\end{align*}
where $\gamma$ is Euler's constant and $\zeta$ is the Riemann zeta-function.
\end{theorem}

In \cite{Young2011a}, Young extended Conrey's result as follows.

\begin{theorem}[Young \cite{Young2011a}]
For primes $p,h$ with $h<p^{1-\varepsilon}$, we have that
\begin{align*}
\frac{p^{1/2}}{\varphi(p)}&\mathcal{S}(p,h)-\frac{h^{1/2}}{\varphi(h)}\mathcal{S}(h,-p)=\frac{p^{1/2}}{h^{1/2}}\left(\log\frac{p}{h}+\gamma-\log(8\pi)\right)\\
&+\zeta\left(\frac{1}{2}\right)^{2}\left(1-2\frac{p^{1/2}}{\varphi(p)}(1-p^{-1/2})+2\frac{h^{1/2}}{\varphi(h)}(1-h^{-1/2})\right)+\mathcal{E}(p,h),
\end{align*}
where $\varphi(p)$ is Euler's totient function and
$$\mathcal{E}(p,h)\ll hp^{-1-\varepsilon}+h^{-C},$$
for all fixed $\varepsilon,C>0$.
\end{theorem}

Advancing the study of twisted moments of Dirichlet $L$-functions, Bettin \cite{Bettin2016} showed that the error term $\mathcal{E}(p,h)$ can be extended to a continuous function with respect to the real topology. In his work, Bettin extended the known reciprocity results for twisted moments by establishing an exact formula with shifts.

Another related problem, that can be seen as a generalization of \eqref{eq:S}, is to obtain asymptotic formulas for the twisted moments of Dirichlet $L$-functions with \textit{two} twists. In other words, the aim is to study
\begin{equation}\label{eq:S2}
M_{\pm}(h,k;p):=\frac{p^{1/2}}{\varphi(p)}\sideset{}{^*}\sum_{\substack{\chi(\bmod p) \\ \chi(-1)=\pm1}}\left|L\left(\frac{1}{2},\chi\right)\right|^{2}\chi(h)\overline{\chi}(k),
\end{equation} 
where $h,p$ and $k$ are prime numbers. The quantity \eqref{eq:S2} was first studied by Selberg \cite{Sel}. In \cite{Bettin2016}, Bettin improved Selberg's result on the second moment of Dirichlet $L$-functions with two twists.

\begin{theorem}[Bettin \cite{Bettin2016}]
Let $h,k$ and $p$ be distinct prime numbers such that $p\geq4hk$. Then
\begin{align*}
M_{\pm}(h,k;p)=&\pm M_{\pm}(h,p;k)\pm M_{\pm}(k,p;h)\\
&+\frac{1}{2}\left(\frac{p}{hk}\right)^{1/2}\left(\log\frac{p}{hk}+\gamma-\log(8\pi)\mp\frac{\pi}{2}\right)+O(\log p).
\end{align*}
\end{theorem} 

More recently, there have been some interesting developments on the study of twisted second moments of Dirichlet $L$-functions over rational function fields. Let $q$ be a power of an odd prime number and $\mathbb{A}=\mathbb{F}_{q}[T]$ the polynomials with coefficients in the finite field $\mathbb{F}_{q}$. In this setting, Djankovi\'{c} \cite{Djankovic2018} proved the following.

\begin{theorem}[Djankovi\'{c} \cite{Djankovic2018}]
Let $P,H$ be irreducible polynomials in $\mathbb{F}_{q}[T]$ and
$$\mathcal{S}(P,H):=\sideset{}{^*}\sum_{\chi(\bmod P)}\left|L\left(\frac{1}{2},\chi\right)\right|^{2}\chi(H).$$
If $H\neq P$ and $\deg(H)\leq\deg(P)$ then
\begin{align*}
\frac{|P|^{1/2}}{\phi(P)}&\mathcal{S}(P,H)-\frac{|H|^{1/2}}{\phi(H)}\mathcal{S}(H,-P)=\frac{|P|^{1/2}}{|H|^{1/2}}\left(\deg(P)-\deg(H)-\zeta_{\mathbb{A}}\left(\frac{1}{2}\right)^{2}\right)\\
&+\zeta_{\mathbb{A}}\left(\frac{1}{2}\right)^{2}\left(1-2\frac{|P|^{1/2}}{\phi(P)}(1-|P|^{-1/2})+2\frac{|H|^{1/2}}{\phi(H)}(1-|H|^{-1/2})\right),
\end{align*}
where $L(s,\chi)$ is the Dirichlet $L$-function in function fields associated to the Dirichlet character $\chi$ modulo $P$, with $\zeta_{\mathbb{A}}(s)$ being the zeta-function for $\mathbb{F}_{q}[T]$, $\phi(P)$ the Euler's totient function for polynomials and $|P|=q^{\deg(P)}$ denotes the norm of a polynomial $P$ in $\mathbb{F}_{q}[T]$.
\end{theorem}

In a recent work, Djankovi\'{c}, {\DJ}oki\'c and Lelas \cite{Djankovic2021a} have established a function field analogue of Bettin's result about twisted second moments of Dirichlet $L$-functions with two twists. If we let $H, K$ and $Q$ be monic irreducible polynomials in $\mathbb{F}_q[T]$ and restrict the sum further to be over all even or odd Dirichlet characters modulo $Q$ then the problem is to establish an asymptotic formula for 
\begin{equation}\label{secondmonenttwotwist}
\mathcal{S}^{\pm}(Q;H,K)=\frac{|Q|^{\frac{1}{2}}}{\phi^{\pm}(Q)}\sideset{}{^{\pm}}\sum_{\substack{\chi (\text{mod }Q)\\\chi\neq\chi_0}}\left|L\left(\frac{1}{2},\chi\right)\right|^2\chi(H)\bar{\chi}(K),
\end{equation}
where $\phi^{\pm}(Q)$ denotes the number of even or odd Dirichlet characters modulo $Q$. Motivated by the methods of Bettin \cite{Bettin2016}, Djankovi\'c, \DJ oki\'c and Lelas \cite{Djankovic2021a} established a triple reciprocity formula involving $\mathcal{S}^{-}(Q;H,K)$, $\mathcal{S}^{-}(H;K,-Q)$ and $\mathcal{S}^{-}(K;H,-Q)$ and involving $\mathcal{S}^+(Q;H,K)$, $\mathcal{S}^+(H;K,Q)$ and $\mathcal{S}^+(K;H,Q)$. In particular, they proved the following results.

\begin{theorem}[Djankovi\'c, \DJ oki\'c and Lelas \cite{Djankovic2021a}]\label{Ch9primetwotwist}
Let $H$, $K$ and $Q$ be distinct monic irreducible polynomials in $\mathbb{F}_q[T]$ such that $\deg(H)+\deg(K)\leq \deg(Q)$. Then we have the following triple reciprocity formulas:
\begin{align*}
\mathcal{S}^{-}(Q;H,K)&=\mathcal{S}^{-}(H;K,-Q)+\mathcal{S}^{-}(K;H,-Q)\\
&+\frac{|Q|^{\frac{1}{2}}}{|HK|^{\frac{1}{2}}}\left(\deg(Q)-\deg(H)-\deg(K)\right)
\end{align*}
and
\begin{align*}
\mathcal{S}^+(Q;H,K)&=\mathcal{S}^+(H;K,Q)+\mathcal{S}^+(K;H,Q)\\
&+\frac{|Q|^{\frac{1}{2}}}{|HK|^{\frac{1}{2}}}\left(\deg(Q)-\deg(H)-\deg(K)-\zeta_{\mathbb{A}}\left(\frac{1}{2}\right)^2(q-1)\right)\\
&-2\zeta_{\mathbb{A}}\left(\frac{1}{2}\right)^2\left(\frac{|Q|^{\frac{1}{2}}-1}{\phi^+(Q)}-\frac{|H|^{\frac{1}{2}}-1}{\phi^+(H)}-\frac{|K|^{\frac{1}{2}}-1}{\phi^+(K)}\right).
\end{align*}
\end{theorem}

The aim of this note is to extend the above results of Djankovi\'c and Djankovi\'c, \DJ oki\'c and Lelas. In their work they only consider Dirichlet characters modulo a monic irreducible polynomial, i.e., they only prove results for prime moduli. In this note we establish results for a general moduli. In particular we prove the following. 

\begin{theorem}\thlabel{onetwist}
Let $H$ and $R$ be monic polynomials in $\mathbb{F}_q[T]$ with $\deg(H)<\deg(R)$. Then 
\begin{equation}\label{eqnmonic1twist}
\frac{1}{\phi^*(R)}\sideset{}{^*}\sum_{\chi(\text{mod }R)}\left|L\left(\frac{1}{2},\chi\right)\right|^2\chi(H)=|H|^{\frac{1}{2}}\frac{\phi(R)}{|R|}\deg(HR)+O\left(|H|^{\frac{1}{2}}\log\omega(R)\right),
\end{equation}
where $\omega(R)$ is the number of distinct prime factors of $R$, $\phi^*(R)$ denotes the number of primitive Dirichlet characters modulo $R$ and the $\ast$ indicates a summation over all primitive Dirichlet characters modulo $R$.
\end{theorem}

And for the two twists we have the following.

\begin{theorem}\thlabel{twotwist}
Let $H$, $K$ and $R$ be monic polynomials in $\mathbb{F}_q[T]$ with $\deg(H)+\deg(K)<\deg(R)$. Then 
\begin{equation}\label{eqnmonic2twist}
\frac{1}{\phi^*(R)}\sideset{}{^*}\sum_{\chi (\text{mod }R)}\left|L\left(\frac{1}{2},\chi\right)\right|^2\chi(H)\bar{\chi}(K)=|HK|^{\frac{1}{2}}\frac{\phi(R)}{|R|}\deg(HKR)+O\left(|HK|^{\frac{1}{2}}\log\omega(R)\right).
\end{equation}
\end{theorem}




\section{A short overview of Dirichlet $L$-functions over function fields}

In this section, we give a short overview of Dirichlet $L$-functions in function fields, with most of these facts stated in \cite{Rosen2002}. Let $\mathbb{F}_q$ denote a finite field with $q$ elements, where $q$ is a power of an odd prime and $\mathbb{A}=\mathbb{F}_q[T]$ be its polynomial ring. Furthermore, we denote by $\mathbb{A}^+$, $\mathbb{A}^+_n$ and $\mathbb{A}^+_{\leq n}$ the set of all monic polynomials in $\mathbb{A}$, the set of all monic polynomials in $\mathbb{A}$ of degree $n$ and the set of all monic polynomials of degree at most $n$ in $\mathbb{A}$ respectively. For $f\in\mathbb{A}$, the norm of $f$, $|f|$, is defined to be equal to $q^{\deg(f)}$ and $\phi(f)$, $\mu(f)$ and $\omega(f)$ denotes the Euler-Totient function for $\mathbb{A}$, the M\"obius function for $\mathbb{A}$ and the number of distinct prime factors of $f$.

For $\Re(s)>1$, the zeta function for $\mathbb{A}$ is defined as 
\begin{equation}
    \zeta_{\mathbb{A}}(s)=\sum_{f\in\mathbb{A}^+}\frac{1}{|f|^s}=\prod_P\left(1-\frac{1}{|P|^s}\right)^{-1},
\end{equation}
where the product is over all monic irreducible polynomials in $\mathbb{A}$. Since there are $q^n$ monic polynomials of degree $n$ in $\mathbb{A}$, then
\begin{equation*}
    \zeta_{\mathbb{A}}(s)=\frac{1}{1-q^{1-s}}.
\end{equation*}

\begin{definition}
    Let $R\in\mathbb{A}^+$. Then a Dirichlet character modulo $R$ is defined to be a function $\chi:\mathbb{A}\rightarrow \mathbb{C}$ which satisfies the following properties:
    \begin{enumerate}
        \item $\chi(AB)=\chi(A)\chi(B), \hspace{0.5cm}\forall A,B\in\mathbb{A}$,
        \item $\chi(A+BR)=\chi(A), \hspace{0.5cm}\forall A,B\in\mathbb{A}$,
        \item $\chi(A)\neq 0 \iff (A,R)=1$.
    \end{enumerate}
\end{definition}
A Dirichlet character $\chi$ is said to be even if $\chi(a)=1$ for all $a\in\mathbb{F}_q^*$. Otherwise we say that it is odd. 

\begin{definition}
    Let $R\in\mathbb{A}^+, S|R$ and $\chi$ be a character of modulus $R$. We say that $S$ is an induced modulus of $\chi$ if there exists a character $\chi_1$ of modulus $S$ such that 
\[
\chi(A)=
\begin{cases}
\chi_1(A)&\text{if }(A,R)=1,\\
0&\text{otherwise}.
\end{cases}
\]
 We say $\chi$ is primitive if there is no induced modulus of strictly smaller norm than $R$. Otherwise $\chi$ is said to be non-primitive. Let $\phi^*(R)$ denote the number of primitive characters of modulus $R$. 
\end{definition}

\begin{definition}
   Let $\chi$ be a Dirichlet character modulo $R$. Then the Dirichlet $L$-function  corresponding to $\chi$ is defined by
\begin{equation}
L(s,\chi):=\sum_{f\in\mathbb{A}^+}\frac{\chi(f)}{|f|^s}
\end{equation}
which converges absolutely for $\Re(s)>1$.  
\end{definition}
To finish this section, we will state some results about multiplicative functions in function fields which will be used throughout this paper. Taking Euler products, we see that for all $s\in\mathbb{C}$ and all $R\in\mathbb{A}$, we have 

\begin{equation}\label{Mobiussum}
\sum_{E|R}\frac{\mu(E)}{|E|^s}=\prod_{P|R}\left(1-\frac{1}{|P|^s}\right)
\end{equation}
and differentiating (\ref{Mobiussum}), we see that for all $s\in\mathbb{C}\backslash\{0\}$, we have 
\begin{equation}\label{Mobiusdiffsum}
\sum_{E|R}\frac{\mu(E)\deg(E)}{|E|^s}=-\left(\prod_{P|R}1-\frac{1}{|P|^s}\right)\left(\sum_{P|R}\frac{\deg(P)}{|P|^s-1}\right).
\end{equation}

\begin{lemma}[{\cite[Lemma~4.5]{Andrade2021a}}]\thlabel{Omegaboundsum}
Let $R\in\mathbb{A}^+$. We have that
\begin{equation}
\sum_{P|R}\frac{\deg(P)}{|P|-1}\ll\log\omega(R).
\end{equation}
\end{lemma}

\begin{lemma}[{\cite[Lemma~A.2.3]{Yiasemides2020}}]\thlabel{omegabound}
For $\deg(R)>1$, we have 
\begin{equation}
\omega(R)\ll\frac{\log_q|R|}{\log_q\log_q|R|},
\end{equation}
where the implied constant is independent of $q$. 
\end{lemma}

\begin{lemma}\thlabel{2omega(R)result}
We have 
\begin{equation}
2^{\omega(R)}=\sum_{E|R}|\mu(E)|.
\end{equation}
Also, for any $\epsilon>0$ we have 
\begin{equation}
    2^{\omega(R)}\ll_{\epsilon}|R|^{\epsilon}.
\end{equation}
\end{lemma}

\begin{lemma}[{\cite[Lemma~A.2.4]{Yiasemides2020}}]\thlabel{phi(R)bound}
For $\deg(R)>q$ we have 
\begin{equation}
\phi(R)\gg\frac{|R|}{\log_q\log_q|R|}.
\end{equation}
\end{lemma} 

\begin{lemma}[{\cite[Lemma~A.2.5]{Yiasemides2020}}]\thlabel{phi*(R)bound}
For $\deg(R)>q$, we have 
\begin{equation}
\phi^*(R)\gg\frac{\phi(R)}{\log_q\log_q|R|}.
\end{equation}
\end{lemma}

\begin{lemma}[{\cite[Lemma~3.7]{Andrade2021a}}] \thlabel{Rmonicorthogrel}
Let $R\in\mathbb{A}^+$ and $A,B\in\mathbb{A}$. Then
\[
\sideset{}{^*}\sum_{\chi (\text{mod }R)}\chi(A)\bar{\chi}(B)=
\begin{cases}
\sum_{\substack{EF=R\\F|(A-B)}}\mu(E)\phi(F)&\text{if }(AB,R)=1,\\
0&\text{otherwise}
\end{cases}
\]
\end{lemma}
As a Corollary we have the following result.

\begin{corollary}[{\cite[Corollary~3.8]{Andrade2021a}}]\thlabel{Rprimmoniccor}
For all $R\in\mathbb{A}^+$ we have that
\begin{equation}
    \phi^*(R)=\sum_{EF=R}\mu(E)\phi(F).
\end{equation}
\end{corollary}

\section{Preliminary Lemmas}

In this section, we state and prove results which will be needed to prove both \thref{onetwist} and \thref{twotwist}. We start by stating the approximate function equation for $\left|L\left(\frac{1}{2},\chi\right)\right|^2$.

\begin{lemma}[{\cite[Lemma~2.5]{Gao2022a}}]\thlabel{AFE}
Let $\chi$ be a primitive Dirichlet character of modulus $R$. Then we have
\begin{equation}
    \left|L\left(\frac{1}{2},\chi\right)\right|^2=2\sum_{\substack{A,B\in\mathbb{A}^+\\\deg(AB)<\deg(R)}}\frac{\chi(A)\bar{\chi}(B)}{|AB|^{\frac{1}{2}}}+O\left(|R|^{-\frac{1}{2}+\epsilon}\right).
\end{equation}
\end{lemma}

The next lemma will be used to obtain the main term of \thref{onetwist} and \thref{twotwist}.

\begin{lemma}\thlabel{maintermlemma}
Let $H$ and $R$ be fixed monic polynomials in $\mathbb{F}_q[T]$ with $\deg(H)<\deg(R)$ and let $x$ be a positive integer. If $x\geq \deg(R)-\deg(H)$, then 
\begin{equation}
    \sum_{\substack{A\in\mathbb{A}^+_{\leq x}\\(AH,R)=1}}\frac{1}{|A|}=|H|\frac{\phi(R)}{|R|}(x+\deg(H))+O\left(|H|\log\omega(R)\right). 
\end{equation}
Whereas if $x<\deg(R)-\deg(H)$, then 
\begin{equation}
  \sum_{\substack{A\in\mathbb{A}^+_{\leq x}\\(AH,R)=1}}\frac{1}{|A|}=|H|\frac{\phi(R)}{|R|}(x+\deg(H))+O\left(|H|\log\omega(R)\right)+O\left( 2^{\omega(R)}q^{-x}(x+\deg(H))\right).
\end{equation}
\end{lemma}

\begin{proof}
    For all positive integers $x$ we have 
\begin{equation}\label{Lemma3.2sum1}
\sum_{\substack{A\in\mathbb{A}^+_{\leq x}\\(AH,R)=1}}\frac{1}{|A|}=\sum_{A\in\mathbb{A}^+_{\leq x}}\frac{1}{|A|}\sum_{E|(AH,R)}\mu(E)=\sum_{A\in\mathbb{A}^+_{\leq x}}\frac{1}{|A|}\sum_{\substack{E|AH\\E|R}}\mu(E)=\sum_{E|R}\mu(E)\sum_{\substack{A\in\mathbb{A}^+_{\leq x}\\E|AH}}\frac{1}{|A|}.
\end{equation}

Since $E|AH$ then $EL=AH$ for some $L\in\mathbb{A}^+$ with $\deg(L)=\deg(A)+\deg(H)-\deg(E)\leq x+\deg(H)-\deg(E)$. Thus
\begin{equation*}
\sum_{\substack{A\in\mathbb{A}^+_{\leq x}\\(AH,R)=1}}\frac{1}{|A|}=|H|\sum_{\substack{E|R\\\deg(E)\leq x+\deg(H)}}\frac{\mu(E)}{|E|}\sum_{\substack{L\in\mathbb{A}^+\\\deg(L)\leq x+\deg(H)-\deg(E)}}\frac{1}{|L|}.
\end{equation*}

We know that, for a non-negative integer $y$, 
\begin{equation*}
\sum_{L\in\mathbb{A}^+_{\leq y}}\frac{1}{|L|}=\sum_{k=0}^yq^{-k}\sum_{L\in\mathbb{A}^+_k}1=\sum_{k=0}^y1=y+1,
\end{equation*}
and so
\begin{align}\label{Ch9maintermsumexpand}
\sum_{\substack{A\in\mathbb{A}^+_{\leq x}\\(AH,R)=1}}\frac{1}{|A|}&=|H|\sum_{\substack{E|R\\\deg(E)\leq x+\deg(H)}}\frac{\mu(E)}{|E|}(x+\deg(H)-\deg(E)+1)\nonumber\\
&=|H|\sum_{E|R}\frac{\mu(E)}{|E|}(x+\deg(H)-\deg(E)+1)\nonumber\\
&-|H|\sum_{\substack{E|R\\\deg(E)>x+\deg(H)}}\frac{\mu(E)}{|E|}(x+\deg(H)-\deg(E)+1).
\end{align}

Using (\ref{Mobiussum}), (\ref{Mobiusdiffsum}) and \thref{Omegaboundsum} we have 
\begin{equation}
\sum_{E|R}\frac{\mu(E)}{|E|}(x+\deg(H)-\deg(E)+1)=\frac{\phi(R)}{|R|}(x+\deg(H))+O\left(\log\omega(R)\right). 
\end{equation}

If $x+\deg(H)\geq \deg(R)$, then there is no $E|R$ with $\deg(E)>\deg(R)$ and so the final term on the right-hand side of (\ref{Ch9maintermsumexpand}) is empty. Thus for $x+\deg(H)\geq \deg(R)$
\begin{equation}
\sum_{\substack{E|R\\\deg(E)>x+\deg(H)}}\frac{\mu(E)}{|E|}(x+\deg(H)-\deg(E)+1)=0.
\end{equation}

Whereas for $x+\deg(H)<\deg(R)$, we have
\begin{align*}
    \sum_{\substack{E|R\\\deg(E)>x+\deg(H)}}\frac{\mu(E)}{|E|}(x+\deg(H)-\deg(E)+1)&\ll \sum_{\substack{E|R\\\deg(E)>x+\deg(H)}}\frac{|\mu(E)|}{|E|}\deg(E)\\
    &\ll \frac{x+\deg(H)}{q^{x+\deg(H)}}\sum_{\substack{E|R\\\deg(E)>x+\deg(H)}}|\mu(E)|\\
    &\ll \frac{2^{\omega(R)}(x+\deg(H))}{q^{x+\deg(H)}},
\end{align*}
here the final inequality follows from \thref{2omega(R)result}. Combining the above completes the proof.
\end{proof}

The following lemmas will be used to create a suitable bound for the error term of \thref{onetwist} and \thref{twotwist}.

\begin{lemma}\thlabel{errorterm1}
Let $F$, $H$ and $R$ be fixed monic polynomials in $\mathbb{F}_q[T]$ where $F|R$ and let $z<\deg(R)$. Then 
\begin{equation}
\sum_{\substack{A,B\in\mathbb{A}^+\\\deg(AB)=z\\AH\equiv B(\text{mod }F)\\AH\neq B\\(ABH,R)=1}}\frac{1}{|AB|^{\frac{1}{2}}}\ll\frac{q^{\frac{z}{2}}(z+1)|H|}{|F|}.
\end{equation}
\end{lemma}

\begin{proof}
    We consider three cases, $\deg(AH)>\deg(B)$, $\deg(AH)<\deg(B)$ and $\deg(AH)=\deg(B)$ where $AH\neq B$.
    
If we first consider the case $\deg(AH)>\deg(B)$ and suppose that $\deg(A)=i$, then since $AH\equiv B(\text{mod }F)$ and $AH\neq B$ we have that $AH=LF+B$ for some $L\in\mathbb{A}^+$ with $\deg(L)=i+\deg(H)-\deg(F)$ and $\deg(B)=z-\deg(A)=z-i$. Thus, combining the above we have 
\begin{align}\label{onetwistboundpart1}
\sum_{\substack{A,B\in\mathbb{A}^+\\\deg(AB)=z\\\deg(AH)>\deg(B)\\AH\equiv B(\text{mod }F)\\AH\neq B\\(ABH,R)=1}}\frac{1}{|AB|^{\frac{1}{2}}}&\leq q^{-\frac{z}{2}}\sum_{i=0}^z\sum_{\substack{L\in\mathbb{A}^+\\\deg=i+\deg(H)-\deg(F)}}\sum_{\substack{B\in\mathbb{A}^+\\\deg(B)=z-i}}1\nonumber\\
&=q^{\frac{z}{2}}\sum_{i=0}^zq^{-i}\sum_{\substack{L\in\mathbb{A}^+\\\deg(L)=i+\deg(H)-\deg(F)}}1=\frac{q^{\frac{z}{2}}|H|}{|F|}\sum_{i=0}^z1=\frac{q^{\frac{z}{2}}(z+1)|H|}{|F|}.
\end{align}
    
Similarly, considering the case $\deg(AH)<\deg(B)$ and using similar arguments seen previously we have 
\begin{align}\label{onetwistboundpart2}
\sum_{\substack{A,B\in\mathbb{A}^+\\\deg(AB)=z\\\deg(B)>\deg(AH)\\AH\equiv B(\text{mod }F)\\AH\neq B\\(ABH,R)=1}}\frac{1}{|AB|^{\frac{1}{2}}}&\leq \frac{q^{\frac{z}{2}}(z+1)}{|F|}.
\end{align}

Finally, if we consider the case where $\deg(AH)=\deg(B)=i$, then $2i=\deg(ABH)=z+\deg(H)$ and so $\deg(B)=i=\frac{z+\deg(H)}{2}$. Furthermore since $AH\equiv B(\text{mod F})$ and $AH\neq B$, then $AH=LF+B$ where $L\in\mathbb{A}$ with $\deg(L)<i-\deg(F)=\frac{z+\deg(H)}{2}-\deg(F)$. Thus combining the above we have 
\begin{align}\label{Ch9onetwistboundpart3}
\sum_{\substack{A,B\in\mathbb{A}^+\\\deg(AB)=z\\\deg(AH)=\deg(B)\\AH\equiv B(\text{mod }F)\\AH\neq B\\(ABH,R)=1}}\frac{1}{|AB|^{\frac{1}{2}}}&\leq q^{-\frac{z}{2}}\sum_{\substack{B\in\mathbb{A}^+\\\deg(B)=\frac{z+\deg(H)}{2}}}\sum_{\substack{L\in\mathbb{A}\\\deg(L)<\frac{z+\deg(H)}{2}-\deg(F)}}1\nonumber\\
&=\frac{|H|^{\frac{1}{2}}}{|F|}\sum_{\substack{B\in\mathbb{A}^+\\\deg(B)=\frac{z+\deg(H)}{2}}}1=\frac{q^{\frac{z}{2}}|H|}{|F|}.
\end{align}
Combining all the cases proves the result. 
\end{proof}

\begin{lemma}\thlabel{errorterm2}
Let $F$, $H$, $K$ and $R$ be fixed monic polynomials in $\mathbb{F}_q[T]$ where $F|R$ and let $z<\deg(R)$. Then
\begin{equation}
\sum_{\substack{A,B\in\mathbb{A}^+\\\deg(AB)=z\\AH\equiv BK(\text{mod }F)\\AH\neq BK\\(ABHK,R)=1}}\frac{1}{|AB|^{\frac{1}{2}}}\ll \frac{q^{\frac{z}{2}}(z+1)|HK|}{|F|}.
\end{equation}
\end{lemma}

\begin{proof}
The proof is similar to the proof of \thref{errorterm1} and \cite[Lemma~6.4]{Yiasemides2021}.
\end{proof}

\begin{lemma}\thlabel{arithmeticbound}
For all $R\in\mathbb{A}^+$ and $\epsilon>0$ we have 
\begin{equation}\label{Ch9errorineq}
\frac{2^{\omega(R)}|R|^{\frac{1}{2}}\deg(R)}{\phi^*(R)}\ll_{\epsilon}|R|^{\epsilon-\frac{1}{2}}.
\end{equation}
\end{lemma}

\begin{proof}
For $\deg(R)\leq q$ we know, by \cite[(A.2.3)]{Yiasemides2020} that $\frac{\phi^*(R)}{|R|}\gg 1$. Thus for $\deg(R)\leq q$ we have 
\begin{equation*}
\frac{2^{\omega(R)}|R|^{\frac{1}{2}}\deg(R)}{\phi^*(R)}\ll\frac{2^{\omega(R)}\deg(R)}{|R|^{\frac{1}{2}}}\ll\frac{2^{\omega(R)}}{|R|^{\frac{1}{2}-\epsilon}}.
\end{equation*}
From \thref{2omega(R)result} we know that $2^{\omega(R)}\ll |R|^{\epsilon}$, thus (\ref{Ch9errorineq}) holds for $\deg(R)\leq q$. \\

\par\noindent
For $\deg(R)>q$ we know by \thref{phi(R)bound} and \thref{phi*(R)bound} that
\begin{equation*}
\phi^*(R)\gg\frac{\phi(R)}{\log_q\log_q|R|}\gg \frac{|R|}{(\log_q\log_q|R|)^2}.
\end{equation*}

\par\noindent
Thus if $\deg(R)>q$, then 
\begin{equation*}
\frac{2^{\omega(R)}|R|^{\frac{1}{2}}\deg(R)}{\phi^*(R)}\ll\frac{2^{\omega(R)}\deg(R)(\log_q\log_q|R|)^2}{|R|^{\frac{1}{2}}}\ll_{\epsilon}\frac{2^{\omega(R)}}{|R|^{\frac{1}{2}-\epsilon}}
\end{equation*}

\par\noindent
Finally, from \thref{2omega(R)result}, we know that $2^{\omega(R)}\ll |R|^{\epsilon}$,  then (\ref{Ch9errorineq}) holds for $\deg(R)>q$ and thus completes the proof. 
\end{proof}

\section{Proof of Theorem 1.6}

In this section, we use results stated previously to prove \thref{onetwist}.

\begin{proof}[Proof of \thref{onetwist}]
Using the approximate function equation \thref{AFE} we have 
\begin{align}
&\frac{1}{\phi^*(R)}\sideset{}{^*}\sum_{\chi(\text{mod }R)}\left|L\left(\frac{1}{2},\chi\right)\right|^2\chi(H)=\frac{2}{\phi^*(R)}\sideset{}{^*}\sum_{\chi(\text{mod }R)}\sum_{\substack{A,B\in\mathbb{A}^+\\\deg(AB)<\deg(R)}}\frac{\chi(A)\bar{\chi}(B)\chi(H)}{|AB|^{\frac{1}{2}}}+O\left(|R|^{-\frac{1}{2}+\epsilon}\right). 
\end{align}

Using the orthogonality relation \thref{Rmonicorthogrel}, we have 
\begin{align}\label{onetwistorthog}
\frac{2}{\phi^*(R)}\sideset{}{^*}\sum_{\chi(\text{mod }R)}\sum_{\substack{A,B\in\mathbb{A}^+\\\deg(AB)<\deg(R)}}\frac{\chi(A)\bar{\chi}(B)\chi(H)}{|AB|^{\frac{1}{2}}}=\frac{2}{\phi^*(R)}\sum_{EF=R}\mu(E)\phi(F)\sum_{\substack{A,B\in\mathbb{A}^+\\ \deg(AB)<\deg(R)\\AH\equiv B(\text{mod }F)\\(ABH,R)=1}}\frac{1}{|AB|^{\frac{1}{2}}}.
\end{align}

For the second sum on the right-hand side of (\ref{onetwistorthog}), we will consider the contribution of the diagonal, $AH=B$, and the off-diagonal, $AH\neq B$, terms separately. Thus we write
\begin{align*}
 \frac{2}{\phi^*(R)}\sum_{EF=R}\mu(E)\phi(F)\sum_{\substack{A,B\in\mathbb{A}^+\\ \deg(AB)<\deg(R)\\AH\equiv B(\text{mod }F)\\(ABH,R)=1}}\frac{1}{|AB|^{\frac{1}{2}}}&=\frac{2}{\phi^*(R)}\sum_{EF=R}\mu(E)\phi(F)\sum_{\substack{A,B\in\mathbb{A}^+\\ \deg(AB)<\deg(R)\\AH\equiv B(\text{mod }F)\\AH=B\\(ABH,R)=1}}\frac{1}{|AB|^{\frac{1}{2}}}\\
&+\frac{2}{\phi^*(R)}\sum_{EF=R}\mu(E)\phi(F)\sum_{\substack{A,B\in\mathbb{A}^+\\ \deg(AB)<\deg(R)\\AH\equiv B(\text{mod }F)\\AH\neq B\\(ABH,R)=1}}\frac{1}{|AB|^{\frac{1}{2}}}.
\end{align*}

Considering the contribution of the diagonal, $AH=B$, the double sum over all $A,B\in\mathbb{A}^+$ with $\deg(AB)<\deg(R)$, $AH=B$ and $(ABH,R)=1$ becomes a single sum over all $A\in\mathbb{A}^+$ with $\deg(A)<\frac{1}{2}(\deg(R)-\deg(H))$ and $(AH,R)=1$. Therefore using the arguments stated above and \thref{Rprimmoniccor} we have 
\begin{equation}
\frac{2}{\phi^*(R)}\sum_{EF=R}\mu(E)\phi(F)\sum_{\substack{A,B\in\mathbb{A}^+\\\deg(AB)<\deg(R)\\AH=B\\(ABH,R)=1}}\frac{1}{|AB|^{\frac{1}{2}}}=\frac{2}{|H|^{\frac{1}{2}}}\sum_{\substack{A\in\mathbb{A}^+\\\deg(A)<\frac{\deg(R)-\deg(H)}{2}\\(AH,R)=1}}\frac{1}{|A|}.
\end{equation}

Using \thref{maintermlemma} with $x=\frac{\deg(R)-\deg(H)}{2}-1$ we have 
\begin{equation}
\frac{2}{|H|^{\frac{1}{2}}}\sum_{\substack{A\in\mathbb{A}^+\\\deg(A)<\frac{\deg(R)-\deg(H)}{2}\\(AH,R)=1}}\frac{1}{|A|}=|H|^{\frac{1}{2}}\frac{\phi(R)}{|R|}(\deg(H)+\deg(R))+O\left(|H|^{\frac{1}{2}}\log\omega(R)\right). 
\end{equation}

For the contribution of the off-diagonal terms we use \thref{errorterm1} to give 
\begin{equation}\label{onetwistoffdiagbound1}
    \sum_{\substack{A,B\in\mathbb{A}^+\\ \deg(AB)<\deg(R)\\AH\equiv B(\text{mod }F)\\AH\neq B\\(ABH,R)=1}}\frac{1}{|AB|^{\frac{1}{2}}}=\sum_{z=0}^{\deg(R)-1}\sum_{\substack{A,B\in\mathbb{A}^+\\\deg(AB)=z\\AH\equiv B(\text{mod }F)\\AH\neq B\\(ABH,R)=1}}\frac{1}{|AB|^{\frac{1}{2}}}\ll \sum_{z=0}^{\deg(R)-1}\frac{|H|q^{\frac{z}{2}}(z+1)}{|F|}\ll \frac{|H||R|^{\frac{1}{2}}\deg(R)}{|F|}.
\end{equation}

Thus using (\ref{onetwistoffdiagbound1}) we have 
\begin{equation}\label{onetwistoffdiagbound2}
\frac{2}{\phi^*(R)}\sum_{EF=R}\mu(E)\phi(F)\sum_{\substack{A,B\in\mathbb{A}^+\\ \deg(AB)<\deg(R)\\AH\equiv B(\text{mod }F)\\AH\neq B\\(ABH,R)=1}}\frac{1}{|AB|^{\frac{1}{2}}}\ll \frac{|H||R|^{\frac{1}{2}}\deg(R)}{\phi^*(R)}\sum_{EF=R}|\mu(E)|\frac{\phi(F)}{|F|}.
\end{equation}

Combining (\ref{onetwistoffdiagbound2}) with \thref{2omega(R)result}, \thref{arithmeticbound}  and the fact that $\frac{\phi(R)}{|R|}\leq 1$ we have 
\begin{equation} \label{onetwistoffdiagbound3}
\frac{2}{\phi^*(R)}\sum_{EF=R}\mu(E)\phi(F)\sum_{\substack{A,B\in\mathbb{A}^+\\ \deg(AB)<\deg(R)\\AH\equiv B(\text{mod }F)\\AH\neq B\\(ABH,R)=1}}\frac{1}{|AB|^{\frac{1}{2}}}\ll \frac{2^{\omega(R)}|H||R|^{\frac{1}{2}}\deg(R)}{\phi^*(R)}\ll |H||R|^{\epsilon-\frac{1}{2}}.
\end{equation}

Since $\deg(H)<\deg(R)$, then there is some $\epsilon>0$ such that $\deg(H)\leq (1-2\epsilon)\deg(R)$. Thus $|H|^{\frac{1}{2}}|R|^{\epsilon-\frac{1}{2}}=q^{\frac{1}{2}\deg(H)+\left(\epsilon-\frac{1}{2}\right)\deg(R)}\leq q^{\frac{1}{2}(1-2\epsilon)\deg(R)+\left(\epsilon-\frac{1}{2}\right)\deg(R)}=1$. Therefore combining the above with (\ref{onetwistoffdiagbound3}), we get
\begin{equation}
\frac{2}{\phi^*(R)}\sum_{EF=R}\mu(E)\phi(F)\sum_{\substack{A,B\in\mathbb{A}^+\\ \deg(AB)<\deg(R)\\AH\equiv B(\text{mod }F)\\AH\neq B\\(ABH,R)=1}}\frac{1}{|AB|^{\frac{1}{2}}}\ll |H|^{\frac{1}{2}}.
\end{equation}

Combining the above completes the proof of \thref{onetwist}.
\end{proof}

\section{Proof of Theorem 1.7}

In this section we use similar methods to that seen in the proof of \thref{onetwist} to prove \thref{twotwist}. 

\begin{proof}[Proof of \thref{twotwist}]
Using the approximate functional equation, \thref{AFE}, we have 
\begin{align}
&\frac{1}{\phi^*(R)}\sideset{}{^*}\sum_{\chi(\text{mod }R)}\left|L\left(\frac{1}{2},\chi\right)\right|^2\chi(H)\bar{\chi}(K)\nonumber\\
=&\frac{2}{\phi^*(R)}\sideset{}{^*}\sum_{\chi(\text{mod }R)}\sum_{\substack{A,B\in\mathbb{A}^+\\\deg(AB)<\deg(R)}}\frac{\chi(A)\bar{\chi}(B)\chi(H)\bar{\chi}(K)}{|AB|^{\frac{1}{2}}}+O\left(|R|^{-\frac{1}{2}+\epsilon}\right). 
\end{align}

Using the orthogonality relation \thref{Rmonicorthogrel}, we have 
\begin{align}\label{twotwistsumsplit}
\frac{2}{\phi^*(R)}\sideset{}{^*}\sum_{\chi(\text{mod }R)}\sum_{\substack{A,B\in\mathbb{A}^+\\\deg(AB)<\deg(R)}}\frac{\chi(A)\bar{\chi}(B)\chi(H)\bar{\chi}(K)}{|AB|^{\frac{1}{2}}}=\frac{2}{\phi^*(R)}\sum_{EF=R}\mu(E)\phi(F)\sum_{\substack{A,B\in\mathbb{A}^+\\\deg(AB)<\deg(R)\\AH\equiv BK(\text{mod }F)\\(ABHK,R)=1}}\frac{1}{|AB|^{\frac{1}{2}}}.
\end{align}

For the second sum on the right-hand side of  (\ref{twotwistsumsplit}) we will consider the contribution of the diagonal, $AH=BK$, and off-diagonal, $AH\neq BK$, terms separately. Thus we write
\begin{align*}
\frac{2}{\phi^*(R)}\sum_{EF=R}\mu(E)\phi(F)\sum_{\substack{A,B\in\mathbb{A}^+\\\deg(AB)<\deg(R)\\AH\equiv BK(\text{mod }F)\\(ABHK,R)=1}}\frac{1}{|AB|^{\frac{1}{2}}}&=\frac{2}{\phi^*(R)}\sum_{EF=R}\mu(E)\phi(F)\sum_{\substack{A,B\in\mathbb{A}^+\\\deg(AB)<\deg(R)\\AH\equiv BK(\text{mod }F)\\AH=BK\\(ABHK,R)=1}}\frac{1}{|AB|^{\frac{1}{2}}}\\
&+\frac{2}{\phi^*(R)}\sum_{EF=R}\mu(E)\phi(F)\sum_{\substack{A,B\in\mathbb{A}^+\\\deg(AB)<\deg(R)\\AH\equiv BK(\text{mod }F)\\AH\neq BK\\(ABHK,R)=1}}\frac{1}{|AB|^{\frac{1}{2}}}.
\end{align*}

Considering the contribution of the diagonal, $AH=BK$, the double sum over all $A,B \in\mathbb{A}^+$ with $\deg(AB)<\deg(R)$, $AH=BK$ and $(ABHK,R)=1$ becomes a single sum over $A \in\mathbb{A}^+$ with $\deg(A)<\frac{1}{2}(\deg(R)+\deg(K)-\deg(H))$ and $(AH,R)=1$. Therefore using the arguments stated above and \thref{Rprimmoniccor} we have 
\begin{equation}
    \frac{2}{\phi^*(R)}\sum_{EF=R}\mu(E)\phi(F)\sum_{\substack{A,B\in\mathbb{A}^+\\\deg(AB)<\deg(R)\\AH=BK\\(ABHK,R)=1}}\frac{1}{|AB|^{\frac{1}{2}}}=\frac{2|K|^{\frac{1}{2}}}{|H|^{\frac{1}{2}}}\sum_{\substack{A\in\mathbb{A}^+\\\deg(A)<\frac{\deg(R)+\deg(K)-\deg(H)}{2}\\(AH,R)=1}}\frac{1}{|A|}.
\end{equation}

Using \thref{maintermlemma} with $x=\frac{1}{2}(\deg(R)+\deg(K)-\deg(H))-1$ we have 
\begin{align*}
    &\frac{2|K|^{\frac{1}{2}}}{|H|^{\frac{1}{2}}}\sum_{\substack{A\in\mathbb{A}^+\\\deg(A)<\frac{\deg(R)+\deg(K)-\deg(H)}{2}\\(AH,R)=1}}\frac{1}{|A|}\\
    &=|HK|^{\frac{1}{2}}\frac{\phi(R)}{|R|}(\deg(R)+\deg(H)+\deg(K))+O\left(|HK|^{\frac{1}{2}}\log\omega(R)\right). 
\end{align*}

For the contribution of the off-diagonal terms we use \thref{errorterm1} to give

\begin{align}\label{twotwisterror1}
      \sum_{\substack{A,B\in\mathbb{A}^+\\\deg(AB)<\deg(R)\\AH\equiv BK(\text{mod }F)\\AH\neq BK\\(ABHK,R)=1}}\frac{1}{|AB|^{\frac{1}{2}}}&=\sum_{z=0}^{\deg(R)-1}\sum_{\substack{A,B\in\mathbb{A}^+\\\deg(AB)=z\\AH\equiv BK(\text{mod }F)\\AH\neq BK\\(ABHK,R)=1}}\frac{1}{|AB|^{\frac{1}{2}}}\nonumber\\
      &\ll \sum_{z=0}^{\deg(R)-1}\frac{|HK|q^{\frac{z}{2}}(z+1)}{|F|}\ll \frac{|HK||R|^{\frac{1}{2}}\deg(R)}{|F|}.
\end{align}

Thus using (\ref{twotwisterror1}) we have 
\begin{equation}\label{twotwisterror2}
  \frac{2}{\phi^*(R)}\sum_{EF=R}\mu(E)\phi(F)\sum_{\substack{A,B\in\mathbb{A}^+\\\deg(AB)<\deg(R)\\AH\equiv BK(\text{mod }F)\\AH\neq BK\\(ABHK,R)=1}}\frac{1}{|AB|^{\frac{1}{2}}}\ll \frac{|HK||R|^{\frac{1}{2}}\deg(R)}{\phi^*(R)}\sum_{EF=R}|\mu(E)|\frac{\phi(F)}{|F|}.
\end{equation}

Combining (\ref{twotwisterror2}) with \thref{2omega(R)result}, \thref{arithmeticbound} and the fact that $\frac{\phi(R)}{|R|}\leq 1$ we have 
\begin{equation}\label{twotwisterror3}
  \frac{2}{\phi^*(R)}\sum_{EF=R}\mu(E)\phi(F)\sum_{\substack{A,B\in\mathbb{A}^+\\\deg(AB)<\deg(R)\\AH\equiv BK(\text{mod }F)\\AH\neq BK\\(ABHK,R)=1}}\frac{1}{|AB|^{\frac{1}{2}}}\ll \frac{2^{\omega(R)}|HK||R|^{\frac{1}{2}}\deg(R)}{\phi^*(R)}.
\end{equation}

Since $\deg(H)+\deg(K)<\deg(R)$, then there is some $\epsilon>0$ such that $\deg(H)+\deg(K)\leq (1-2\epsilon)\deg(R)$. Thus $|HK|^{\frac{1}{2}}|R|^{\epsilon-\frac{1}{2}}=q^{\deg(H)+\deg(K)+(\epsilon-\frac{1}{2})\deg(R)}\leq q^{(1-2\epsilon)\deg(R)+(\epsilon-\frac{1}{2})\deg(R)}=1$. Thus, combining the above and (\ref{twotwisterror3}) we have 
\begin{equation}
 \frac{2}{\phi^*(R)}\sum_{EF=R}\mu(E)\phi(F)\sum_{\substack{A,B\in\mathbb{A}^+\\\deg(AB)<\deg(R)\\AH\equiv BK(\text{mod }F)\\AH\neq BK\\(ABHK,R)=1}}\frac{1}{|AB|^{\frac{1}{2}}}\ll |HK|^{\frac{1}{2}}.
\end{equation}
Combining everything completes the proof of \thref{twotwist}.

\end{proof}

\noindent \textbf{Acknowledgment:} The authors are grateful to the Leverhulme Trust (RPG-2017-320) for the support through the research project grant ``Moments of $L$-functions in Function Fields and Random Matrix Theory". The authors also would like to thank Prof. Steve Gonek and Dr. Gihan Marasingha for helpful comments and suggestions on a previous version of this note.


\end{document}